\documentclass{article}%
\usepackage{graphicx}
\usepackage{amsmath}
\usepackage{amsfonts}
\usepackage{amssymb}%
\providecommand{\U}[1]{\protect\rule{.1in}{.1in}}
\providecommand{\U}[1]{\protect\rule{.1in}{.1in}}
\providecommand{\U}[1]{\protect\rule{.1in}{.1in}}
\providecommand{\U}[1]{\protect\rule{.1in}{.1in}}
\providecommand{\U}[1]{\protect\rule{.1in}{.1in}}
\providecommand{\U}[1]{\protect\rule{.1in}{.1in}}
\providecommand{\U}[1]{\protect\rule{.1in}{.1in}}
\providecommand{\U}[1]{\protect\rule{.1in}{.1in}}
\providecommand{\U}[1]{\protect\rule{.1in}{.1in}}
\providecommand{\U}[1]{\protect\rule{.1in}{.1in}}
\providecommand{\U}[1]{\protect\rule{.1in}{.1in}}
\providecommand{\U}[1]{\protect\rule{.1in}{.1in}}
\providecommand{\U}[1]{\protect\rule{.1in}{.1in}}
\providecommand{\U}[1]{\protect\rule{.1in}{.1in}}
\providecommand{\U}[1]{\protect\rule{.1in}{.1in}}
\providecommand{\U}[1]{\protect\rule{.1in}{.1in}}
\providecommand{\U}[1]{\protect\rule{.1in}{.1in}}
\providecommand{\U}[1]{\protect\rule{.1in}{.1in}}
\providecommand{\U}[1]{\protect\rule{.1in}{.1in}}
\providecommand{\U}[1]{\protect\rule{.1in}{.1in}}
\providecommand{\U}[1]{\protect\rule{.1in}{.1in}}
\providecommand{\U}[1]{\protect\rule{.1in}{.1in}}
\providecommand{\U}[1]{\protect\rule{.1in}{.1in}}
\providecommand{\U}[1]{\protect\rule{.1in}{.1in}}
\providecommand{\U}[1]{\protect\rule{.1in}{.1in}}
\providecommand{\U}[1]{\protect\rule{.1in}{.1in}}
\providecommand{\U}[1]{\protect\rule{.1in}{.1in}}
\providecommand{\U}[1]{\protect\rule{.1in}{.1in}}
\providecommand{\U}[1]{\protect\rule{.1in}{.1in}}

\newtheorem{theorem}{Theorem}
{}

\newtheorem{case}{Case}

\newtheorem{corollary}{Corollary}

\newtheorem{definition}{Definition}

{}
\newtheorem{notation}{Notation}

\newtheorem{proposition}{Proposition}
\newtheorem{remark}{Remark}

\newtheorem{summary}{Summary}
\newenvironment{proof}[1][Proof]{\textbf{#1.} }{\ \rule{0.5em}{0.5em}}

\begin{document}

\title{On the Finite-Zone PT-symmetric Dirac Operators}
\author{O. A. Veliev\\{\small \ Department of Mechanical Engineering, Dogus University, Istanbul,
Turkey }\\\ {\small e-mail: oveliev@dogus.edu.tr}}
\date{}
\maketitle

\begin{abstract}
In this paper we find conditions on the PT-symmetric non-self-adjoint Dirac
operator $L(Q)$ with $2\times2$ matrix-valued coefficients $Q$ whose entries
are PT-symmetric periodic functions for which the number of gaps in the real
part of its spectrum is finite.



\end{abstract}

\section{ Introduction}

In this paper, we consider the spectrum of the one-dimensional Dirac operator
$L(P)$ generated in the space $L_{2}^{2}(-\infty,\infty)$ by the differential
expression
\begin{equation}
l(P)=\left(
\begin{array}
[c]{cc}%
-i & 0\\
0 & i
\end{array}
\right)  \mathbf{y}^{^{\prime}}(x)+P\mathbf{y}(x), \tag{1}%
\end{equation}
where $P(x)=\left(
\begin{array}
[c]{cc}%
p_{1}(x) & p_{2}(x)\\
p_{3}(x) & p_{4}(x)
\end{array}
\right)  $ and $p_{j}$ for $j=1,2,3,4$ are the complex-valued functions
satisfying the conditions
\begin{equation}
p_{j}\in L_{1}[0,\pi],\text{ }p_{j}\left(  x+\pi\right)  =p_{j}\left(
x\right)  ,\text{ }p_{j}\left(  -x\right)  =\overline{p_{j}\left(  x\right)
}, \tag{2}%
\end{equation}
for all $x\in\mathbb{R}$. Here $\mathbf{y}=(y_{1},y_{2})^{T}$ is a
vector-valued function. For $-\infty\leq a<b\leq\infty,$ $L_{2}^{2}(a,b)$
denotes the space of vector-valued functions $\mathbf{f}=\left(  f_{1}%
,f_{2}\right)  ^{T}$ equipped with the norm and inner product%
\[
\left\Vert \mathbf{f}\right\Vert =\left(  \int\nolimits_{(a,b)}\left\vert
\mathbf{f}\left(  x\right)  \right\vert ^{2}dx\right)  ^{\frac{1}{2}},\text{
}(\mathbf{f,g})=\int\nolimits_{(a,b)}\left\langle \mathbf{f}\left(  x\right)
,\mathbf{g}\left(  x\right)  \right\rangle dx,
\]
where $\left\vert \cdot\right\vert $ and $\left\langle \cdot,\cdot
\right\rangle $ denote the norm and inner product in $\mathbb{C}^{2},$ respectively.

This paper can be considered as continuation of [15]. In [15], we first proved
that $L(P)$ is a PT-symmetric operator, where we used the following definition
(see for example [1]).

\begin{definition}
We say that the operator $L$ is PT-symmetric if
\[
PTL\mathbf{y}=LPT\mathbf{y}%
\]
for all $\mathbf{y}\in D(L),$ where the space-reflection (parity) operator $P$
and the complex-conjugation operator $T$ are defined by $(P\mathbf{y}%
)(x)=\mathbf{y}(-x)$ and $(T\mathbf{y})(x)=\overline{\mathbf{y}(x)}.$
\end{definition}

We then proved that the main part of the spectrum $\sigma(L(P))$ is real and
contains a large part of the real axis. In the present paper, we obtain
necessary and sufficient conditions for the number of gaps in $\sigma(L(P))$
to be finite. Note that if (2) holds, then, as observed in [15], the spectrum
of $L(P)$ is a shift of the spectrum of the operator $L(Q)$, by some real
number $\gamma,$ where $Q$ has the form
\begin{equation}
Q(x)=\left(
\begin{array}
[c]{cc}%
0 & q_{2}(x)\\
q_{3}(x) & 0
\end{array}
\right)  ,\text{ }q_{j}\in L_{1}[0,\pi],\text{ }q_{j}\left(  x+\pi\right)
=q_{j}\left(  x\right)  ,\text{ }q_{j}\left(  -x\right)  =\overline
{q_{j}\left(  x\right)  }.\tag{3}%
\end{equation}
Therefore, it is sufficient to consider the operator $L(Q).$

It is well-known [5-7, 10] that the spectrum $\sigma(L(Q))$ of the operator
$L(Q)$ is the union of the spectra $\sigma(L_{t}(Q))$ of the operators
$L_{t}(Q),$ for $t\in(-1,1],$ generated in $L_{2}^{2}[0,\pi]$ by (1) and the
quasiperiodic boundary condition
\begin{equation}
\mathbf{y}(\pi)=e^{i\pi t}\mathbf{y}(0).\tag{4}%
\end{equation}
The spectrum of $L_{t}(Q)$ consists of eigenvalues, which are called the Bloch
eigenvalues of $L(Q).$ The eigenvalues of $L_{0}(Q)$ and $L_{1}(Q)$ are called
periodic and antiperiodic eigenvalues respectively. Both are referred to as
$2\pi$- periodic eigenvalues.

Let us briefly describe the organization of this paper. In Section 2, using
some results of [15], formulated as Summary 1, we obtain asymptotic estimates
for all Bloch eigenvalues. For this purpose, we use the asymptotic formulas
for the matrix
\begin{equation}
E(x,\lambda)=\left(
\begin{array}
[c]{cc}%
e_{11}(x,\lambda) & e_{12}(x,\lambda)\\
e_{21}(x,\lambda) & e_{22}(x,\lambda)
\end{array}
\right)  \tag{5}%
\end{equation}
of the fundamental system of solutions for the equation
\begin{equation}
\left(
\begin{array}
[c]{cc}%
-i & 0\\
0 & i
\end{array}
\right)  \mathbf{y}^{^{\prime}}(x)+Q(x)\mathbf{y}(x)=\lambda\mathbf{y}%
(x),\tag{6}%
\end{equation}
with the initial conditions $E(0,\lambda)=I,$ obtained in [8]. In [8], it was
proved that
\begin{align}
e_{11}(x,\lambda) &  =e^{i\lambda x}+\rho_{11}(x,\lambda),\text{ }%
e_{12}(x,\lambda)=\rho_{11}(x,\lambda),\text{ }\tag{7}\\
e_{21}(x,\lambda) &  =\rho_{21}(x,\lambda),\text{ }e_{22}(x,\lambda
)=e^{-i\lambda x}+\rho_{22}(x,\lambda)\nonumber
\end{align}
and $\rho_{kl}(x,\lambda)\rightarrow0,$ as $\lambda\rightarrow\infty$ \ for
$k,l\in\left\{  1,2\right\}  .$ Moreover, these estimates hold uniformly with
respect to $x\in\lbrack0,\pi]$ and for $\left\vert \operatorname{Im}%
\lambda\right\vert <a,$ where $a$ is a fixed positive number.

It is easy to verify that
\[
E^{\prime}(x,\lambda)=A(x,\lambda)E(x,\lambda),
\]
where $A(x,\lambda)=\left(
\begin{array}
[c]{cc}%
i\lambda & iq_{2}(x)\\
-iq_{3}(x) & -i\lambda
\end{array}
\right)  $. Since the trace of $A(x,\lambda)$ is zero, Liouville's formula
yields
\begin{equation}
\det E(\pi,\lambda)=\det E(0,\lambda)=1. \tag{8}%
\end{equation}
A number $\lambda$ is an eigenvalue of the operator $L_{t}\left(  Q\right)  $
if and only if it is a root of the characteristic equation
\[
\Delta(\lambda,t)=\left\vert
\begin{array}
[c]{cc}%
e_{11}(\pi,\lambda)-e^{i\pi t} & e_{21}(\pi,\lambda)\\
e_{12}(\pi,\lambda) & e_{22}(\pi,\lambda)-e^{i\pi t}%
\end{array}
\right\vert =0.
\]
Using (8), we obtain
\[
\Delta(\lambda,t)=e^{i2\pi t}-e^{i\pi t}(e_{11}(\pi,\lambda)+e_{22}%
(\pi,\lambda))+1=0
\]
and hence
\begin{equation}
F(\lambda)=2\cos\pi t, \tag{9}%
\end{equation}
where $F(\lambda)=e_{11}(\pi,\lambda)+e_{22}(\pi,\lambda)$ is an entire function.

In Section 3, we consider the potentials $Q$ for which the real part
$\operatorname{Re}\left(  \sigma(L(Q))\right)  $ of $\sigma(L(Q))$\ has at
most finitely many gaps. Note that if the entries of $Q$\ are real-valued
functions, such potentials are called finite-zone potentials. Thus, in this
paper, we consider finite-zone PT-symmetric periodic potentials for the Dirac operator.

We do not discuss real finite-zone potentials, since the results and methods
developed for the self-adjoint case are not used here. However, to illustrate
the diversity of finite-zone PT-symmetric potentials and real finite-zone
potentials, let us describe in terms of the $2\pi$-periodic eigenvalues.

It is well known and clear that the spectrum of a self-adjoint Dirac operator
has finitely many gaps if and only if all but finitely many of its $2\pi
$-periodic eigenvalues are double eigenvalues. In this paper, we prove that
the spectrum of a PT-symmetric non-self-adjoint Dirac operator has finitely
many gaps if and only if all but finitely many of its $2\pi$-periodic
eigenvalues are either double eigenvalues (Case 1) or nonreal eigenvalues
(Case 2).

Thus, for a PT-symmetric non-self-adjoint Dirac operator, there is an
additional possibility (Case 2),which occurs for a large class of
complex-valued potentials. This provides a way to construct a large class of
PT-symmetric finite-zone potentials.

\section{On the Localizations of the Eigenvalues}

In this section, we consider asymptotic estimates for the eigenvalues of
$L_{t}\left(  Q\right)  $. It follows from (9) that $\sigma(L_{-t}%
(Q))=\sigma(L_{t}(Q))$ and hence
\begin{equation}
\sigma(L(Q))=%
%TCIMACRO{\tbigcup \limits_{t\in\lbrack0,1]}}%
%BeginExpansion
{\textstyle\bigcup\limits_{t\in\lbrack0,1]}}
%EndExpansion
\sigma(L_{t}(Q)).\tag{10}%
\end{equation}
Therefore, it is sufficient to consider the case $t\in\lbrack0,1].$ In [15],
we proved the following result.

\begin{summary}
$(a)$There exists a positive constant $N$ such that the operator $L_{t}\left(
Q\right)  $ has a unique eigenvalue, counting multiplicities, inside each of
the circles $C(2n\pm t,\varepsilon_{n})$, for $t\in T(n)$ and $\left\vert
n\right\vert \geq N,$ where $T(n):=[\varepsilon_{n},1-\varepsilon_{n}]$,
$\varepsilon_{n}\rightarrow0$ as $n\rightarrow\infty$ and $C(a,r)=\left\{
\lambda\in\mathbb{C}:\left\vert \lambda-a\right\vert =r\right\}  .$ In this
paper, we denote by $\lambda_{n}^{-}(t)$ and $\lambda_{n}^{+}(t),$
respectively, the eigenvalues \ lying inside $C(2n-t,\varepsilon_{n})$ and
$C(2n+t,\varepsilon_{n})$. These eigenvalues are the real and simple
eigenvalues of $L_{t}\left(  Q\right)  .$ Therefore, the intervals
\[
\Gamma_{n}^{-}(\varepsilon_{n})=\left\{  \lambda_{n}^{-}(t):t\in
\lbrack\varepsilon_{n},1-\varepsilon_{n}]\right\}  \text{ }\And\text{ }%
\Gamma_{n}^{+}(\varepsilon_{n})=\left\{  \lambda_{n}^{+}(t):t\in
\lbrack\varepsilon_{n},1-\varepsilon_{n}]\right\}
\]
belong to the spectrum of $L(Q).$ Moreover,
\begin{equation}
\lambda\in\sigma(L_{t}(Q))\Longrightarrow\overline{\lambda}\in\sigma
(L_{t}(Q)). \tag{11}%
\end{equation}

$(b)$ For $\left\vert n\right\vert \geq N,$ the intervals $[2n+2\varepsilon
_{n},2n+1-2\varepsilon_{n}]$ and $[2n+1+2\varepsilon_{n},2n+2-2\varepsilon
_{n}]$ belong to the spectrum of $L\left(  Q\right)  .$
\end{summary}

The following statement follows from Summary 1$(b).$

\begin{proposition}
The spectrum of $L(Q)$ has finitely many gaps if there exists $N_{0}\geq N,$
such that the intervals
\begin{equation}
\text{ }(2n-2\varepsilon_{n-1},2n+2\varepsilon_{n})\text{ }\And\text{
}(2n+1-2\varepsilon_{n},2n+1+2\varepsilon_{n})\text{ } \tag{12}%
\end{equation}
for $\left\vert n\right\vert >N_{0}$ belong to $\sigma(L\left(  Q\right)  ).$
\end{proposition}

Therefore, in order to consider the finite-zone Dirac operators, we consider
those potentials $Q$ for which there exists $N_{0}\geq N,$ such that the
intervals (12) for $\left\vert n\right\vert >N_{0}$ belong to $\sigma(L\left(
Q\right)  ).$ We consider the case $n>N_{0}.$ The case $n<-N_{0}$ is similar.
It readily follows from Summary 1 that
\begin{equation}
\lambda_{n}^{-}(t)<\lambda_{n}^{+}(t)\And\lambda_{n}^{+}(t)<\lambda_{n+1}%
^{-}(t) \tag{13}%
\end{equation}
for all $t\in\lbrack\varepsilon_{n},1-\varepsilon_{n}]$ and $n\geq N.$

Since $\lambda_{n}^{-}(t)$ and $\lambda_{n}^{+}(t)$ are simple eigenvalues we
have the following statement.

\begin{theorem}
If $n>N,$ where $N$ is defined in Summary 1, then
\begin{equation}
\Gamma_{n}^{+}(\varepsilon_{n})=[\lambda_{n}^{+}(\varepsilon_{n}),\lambda
_{n}^{+}(1-\varepsilon_{n})]\text{ }\And\Gamma_{n}^{-}(\varepsilon
_{n})=[\lambda_{n}^{-}(1-\varepsilon_{n}),\lambda_{n}^{-}(\varepsilon
_{n})].\text{ } \tag{14}%
\end{equation}

\end{theorem}

\begin{proof}
\textbf{First Proof. }Let us prove the first equality in (14). Since
$\lambda_{n}^{+}(t)$ is a simple eigenvalue of $L_{t}(Q),$ we have
\begin{equation}
F^{^{\prime}}(\lambda_{n}^{+}(t))\neq0, \tag{15}%
\end{equation}
where $F^{^{\prime}}(\lambda)$ is the derivative of $F^{^{\prime}}$ with
respect to $\lambda.$ For $t\in\lbrack\varepsilon_{n},1-\varepsilon_{n}],$
Summary 1 implies that $\lambda_{n}^{+}(t)\in\mathbb{R}.$ Thus, $\lambda
_{n}^{+}(t)$ can be regarded as real-valued differentiable functions of the
real variable $t\in\lbrack\varepsilon_{n},1-\varepsilon_{n}]$ . Since
$F(\lambda_{n}^{+}(t))=2\cos\pi t,$ differentiation with respect to $t$ gives
\[
\frac{d\lambda_{n}^{+}(t)}{dt}=\frac{-2\pi\sin\pi t}{F^{^{\prime}}(\lambda
_{n}^{+}(t))}\neq0,
\]
for $t\in\lbrack\varepsilon_{n},1-\varepsilon_{n}].$ Moreover, it follows from
Summary 1 that
\begin{equation}
\lambda_{n}^{+}(\varepsilon_{n})<\lambda_{n}^{+}(1-\varepsilon_{n}). \tag{16}%
\end{equation}
Therefore, $\lambda_{n}^{+}(t)$ is an increasing continuous function on
$[\varepsilon_{n},1-\varepsilon_{n}].$ Thus the first equality of (14) holds.
In the same way, we prove the second equality of (14).

\textbf{Second proof. }The function $\lambda_{n}^{+}(t)$ is continuous on
$[\varepsilon_{n},1-\varepsilon_{n}].$ Moreover, it is one-to-one on the
interval $[\varepsilon_{n},1-\varepsilon_{n}]$. Indeed suppose that $t_{1}\neq
t_{2}$ and $\lambda_{n}^{+}(t_{1})=\lambda_{n}^{+}(t_{2}).$ Then using (9), we
obtain $\cos\pi t_{1}=\cos\pi t_{2}$ which is a contradiction for $t_{1}%
,t_{2}\in\lbrack\varepsilon_{n},1-\varepsilon_{n}]$ with $t_{1}\neq t_{2}$.
Therefore, $\lambda_{n}^{+}(t)$ is strictly monotone on $[\varepsilon
_{n},1-\varepsilon_{n}]$. On the other hand, (16) holds. Thus, $\lambda
_{n}^{+}(t)$ is an increasing function on $[\varepsilon_{n},1-\varepsilon
_{n}]$ and $\lambda_{n}^{+}([\varepsilon_{n},1-\varepsilon_{n}])=$
$[\lambda_{n}^{+}(\varepsilon_{n}),\lambda_{n}^{+}(1-\varepsilon_{n})].$ In
the same way, we prove the second equality of (14).
\end{proof}

Now, we consider the eigenvalues of $L_{t}\left(  Q\right)  $ for $t\in
\lbrack0,\varepsilon_{n})$ and $t\in(1-\varepsilon_{n},1],$ by using (9). It
follows from the asymptotic formulas (7) that
\begin{equation}
F(\lambda)=2\cos\lambda\pi+\rho(\lambda) \tag{17}%
\end{equation}
and $\rho(\lambda)\rightarrow0$ as $\lambda\rightarrow\infty$ with $\left\vert
\operatorname{Im}\lambda\right\vert <a.$ Let $h$ be fixed positive number such
that $h<\min\left\{  \frac{1}{10\pi},\frac{1}{3}a\right\}  $. There exists
$N(h)>N$ such that
\begin{equation}
\varepsilon_{n}<h,\text{ }\left\vert \rho(\lambda)\right\vert <h^{2}, \tag{18}%
\end{equation}
for $\lambda\in D(h):=\left\{  \lambda\in\mathbb{C}:\operatorname{Re}%
\lambda\in\lbrack2n-1,2n+2],\text{ }\left\vert \operatorname{Im}%
\lambda\right\vert \leq a\right\}  $ and $n\geq N(h)$, where $N$ is defined in
Summary 1. Now, using (9), (17), (18) and Rouch\'{e}'s theorem, we prove the
following result.

\begin{theorem}
$(a)$ If $t\in\lbrack0,h],$ then the operator $L_{t}\left(  Q\right)  $ has
two eigenvalues, counting multiplicities, inside the circle $C(2n,3h)$ for
$n>N(h).$

$(b)$ If $t\in\lbrack1-h,1],$ then the operator $L_{t}\left(  Q\right)  $ has
two eigenvalues, counting multiplicities, inside the circle $C(2n+1,3h)$ for
$n>N(h).$
\end{theorem}

\begin{proof}
$(a)$ To apply Rouche's theorem, we first prove the inequality
\begin{equation}
\left\vert \rho(\lambda)\right\vert <2\left\vert \cos\pi\lambda-\cos
\pi(2n+t)\right\vert \tag{19}%
\end{equation}
for $\lambda=2n+3he^{i\alpha}$, $\alpha\in\lbrack0,2\pi),$ and $t\in
\lbrack0,h].$ Using the trigonometric identity
\[
\cos x-\cos y=-2\sin\frac{x+y}{2}\sin\frac{x-y}{2},
\]
we obtain
\begin{equation}
2\left\vert \cos\pi\lambda-\cos\pi(2n+t)\right\vert =4\left\vert \sin\frac
{\pi}{2}(3he^{i\alpha}+t)\right\vert \left\vert \sin\frac{\pi}{2}%
(3he^{i\alpha}-t)\right\vert . \tag{20}%
\end{equation}
We now estimate the right-hand side of (20). Since $t\in\lbrack0,h],$ and
$h<\frac{1}{10\pi},$ the reverse triangle inequality and the triangle
inequality give
\[
0<\pi h=2h\frac{\pi}{2}\leq\frac{\pi}{2}\left\vert 3he^{i\alpha}\pm
t\right\vert \leq4h\frac{\pi}{2}=2\pi h<\frac{1}{5}.
\]
Therefore, using the elementary inequality $\left\vert \sin z\right\vert
>\frac{1}{2}\left\vert z\right\vert $ for $0<\left\vert z\right\vert <1/5,$ we
see that%
\[
\left\vert \sin\frac{\pi}{2}(3he^{i\alpha}\pm t)\right\vert >\frac{\pi}{2}h.
\]
Substituting these estimates into (20), we obtain%
\[
2\left\vert \cos\pi\lambda-\cos\pi(2n+t)\right\vert >\pi^{2}h^{2}.
\]
Hence, (19) follows from (18) and the last inequality. Consequently, by the
Rouche's theorem and (17), equations (9) and
\begin{equation}
\cos\pi\lambda-\cos\pi t=0 \tag{21}%
\end{equation}
have the same number of roots, counting multiplicity, inside $C(2n,3h),$ where
$n>N(h).$ Since equation (21) has $2$ roots inside $C(2n,3h),$ equation (9)
also has$\ 2$ roots inside this circle, counting multiplicity. This proves
part $(a)$. In the same way, we prove part $(b).$
\end{proof}

\section{Main Results}

Now, using Theorems 1 and 2, we prove the main result of this paper for the
operator $L\left(  Q\right)  $. \ In this section, we determine conditions on
the periodic and antiperiodic eigenvalues under which the intervals (12) (see
Proposition 1) are covered by the eigenvalues of $L_{t}\left(  Q\right)  $ for
$t\in\lbrack0,1]$. For the convenience in the proofs of the next theorems, we
introduce the following notations.

\begin{notation}
Recall that, in Summary 1, for $t\in\lbrack\varepsilon_{n},1-\varepsilon
_{n}],$ the eigenvalues of $L_{t}(Q)$\ lying inside the circles
$C(2n-t,\varepsilon_{n})$ and $C(2n+t,\varepsilon_{n})$ are denoted by
$\lambda_{n}^{-}(t)$ and $\lambda_{n}^{+}(t),$ respectively, where $n\geq N.$
Moreover, (13) holds. We now introduce notation for the eigenvalues of
$L_{t}(Q),$ considered in Theorem 2 for $t\in\lbrack0,h),$ in accordance with
the notations in Summary 1. Since $\varepsilon_{n}<h$ (see (18)), the circles
$C(2n+t,\varepsilon_{n})$ and $C(2n-t,\varepsilon_{n})$ for $t\in\lbrack0,h)$
lies inside $C(2n,3h).$ This means that $\lambda_{n}^{-}(t)$ and $\lambda
_{n}^{+}(t),$ for $t\in\lbrack\varepsilon_{n},h],$ are the eigenvalues of
$L_{t}(Q)$ lying in $C(2n,3h).$ They are real and simple and are precisely the
eigenvalues considered in Theorem 2$(a)$. However, for $t\in\lbrack
0,\varepsilon_{n}),$ the operator $L_{t}(Q)$ may have double eigenvalues.
Since the entire function $F^{^{\prime}}(\lambda)$ has only finitely many
zeros inside the circle $C(2n,3h)$, there exist at most finitely many
$t_{1},t_{2},...,t_{k}\in\lbrack0,h)$ for which the eigenvalues of $L_{t_{j}%
}(Q),$ $j=1,2,...,k,$ lying inside $C(2n,3h)$ are double eigenvalues. Let
\begin{equation}
t_{0}:=0\leq t_{1}<t_{2}<...<t_{k}<h=:t_{k+1}. \tag{22}%
\end{equation}
If $t\in\lbrack0,h)\backslash\left\{  t_{1},t_{2},...,t_{k}\right\}  ,$ then
the circle $C(2n,3h)$ contains exactly two simple eigenvalues of $L_{t}(Q)$.
We denote these eigenvalues again by $\lambda_{n}^{-}(t)$ and $\lambda_{n}%
^{+}(t)$ as follows. If they are real, then $\lambda_{n}^{-}(t)<\lambda
_{n}^{+}(t).$ If they are not real, then they are complex conjugates (see
(11)). In this case, we denote them so that $\operatorname{Im}\lambda_{n}%
^{-}(t)<0$ and $\operatorname{Im}\lambda_{n}^{+}(t)>0.$ If $t=t_{j}$ for some
$j,$ then $\lambda_{n}^{-}(t_{j})=\lambda_{n}^{+}(t_{j}).$

Since $[h,1-h]\subset\lbrack\varepsilon_{n},1-\varepsilon_{n}]$ for $n>N(h)$,
where $h$ and $N(h)$ are defined in the previous section, the notations of
Summary 1 and (13) can be used for $t\in\lbrack h,1-h].$
\end{notation}

Now, we consider the functions $\lambda_{n}^{-}(t)$ and $\lambda_{n}^{+}(t)$
defined in Notation 1 on the interval $[0,h].$

\begin{theorem}
If $n>N(h),$ then the eigenvalues $\lambda_{n}^{-}(t)$ and $\lambda_{n}%
^{+}(t)$ of $L_{t}(Q)$ can be chosen as continuous function of $t$ on $[0,h].$
\end{theorem}

\begin{proof}
First, consider the functions $\lambda_{n}^{-}(t)$ and $\lambda_{n}^{+}(t)$ in
the intervals
\begin{equation}
(t_{0},t_{1}),(t_{1},t_{2}),(t_{2},t_{3}),...,(t_{k-1},t_{k}),(t_{k},t_{k+1})
\tag{23}%
\end{equation}
If $t_{1}=0,$ then the first interval in (23) is absent. Suppose that
$t\in(t_{s},t_{s+1})$ for some $s=0,1,...,k.$ Then $\lambda_{n}^{-}(t)$ and
$\lambda_{n}^{+}(t)$ are simple eigenvalues, and either

$(a)$ $\lambda_{n}^{-}(t)$ and $\lambda_{n}^{+}(t)$ are real numbers
satisfying the inequality $\lambda_{n}^{-}(t)<\lambda_{n}^{+}(t),$

or

$(b)$ $\lambda_{n}^{-}(t)$ and $\lambda_{n}^{+}(t)$ are complex conjugates,
with $\operatorname{Im}\lambda_{n}^{-}(t)<0$ and $\operatorname{Im}\lambda
_{n}^{+}(t)>0.$

Since $\lambda_{n}^{-}(t)$ and $\lambda_{n}^{+}(t)$ are simple eigenvalues,
there exist $\varepsilon>0$ and two functions $\lambda_{1}(t)$ and
$\lambda_{2}(t)$ such that:

$(i)$ $(t-\varepsilon,t+\varepsilon)\subset(t_{s},t_{s+1})$ and the functions
$\lambda_{1}$ and $\lambda_{2}$ are holomorphic on $(t-\varepsilon
,t+\varepsilon);$

$(ii)$ $\lambda_{1}(v)$ and $\lambda_{2}(v)$ are eigenvalues of $L_{v}(Q)$ for
$v\in$ $(t-\varepsilon,t+\varepsilon),$ and $\lambda_{1}(t)=\lambda_{n}%
^{-}(t),$ $\lambda_{2}(t)=\lambda_{n}^{+}(t).$

Moreover, in case $(a),$ $\lambda_{1}(v)<\lambda_{2}(v)$ and in case $(b),$
$\operatorname{Im}\lambda_{2}(v)>0$, for $v\in(t-\varepsilon,t+\varepsilon).$
Therefore, by Notation 1, $\lambda_{n}^{-}(v)=\lambda_{1}(v)$ and $\lambda
_{n}^{+}(t)=\lambda_{2}(v)$ for $v\in(t-\varepsilon,t+\varepsilon).$ Thus,
$\lambda_{n}^{-}(t)$ and $\lambda_{n}^{+}(t)$ are holomorphic functions on
each of the intervals in (23).

Let $\lambda^{-}$ and $\lambda^{+}$ be arbitrary limit point of $\lambda
_{n}^{-}(t)$ and $\lambda_{n}^{+}(t),$ respectively, as $t\rightarrow t_{s}.$
Since $F$ is continuous functions, it follows from (9) that
\[
F(\lambda^{\pm})=2\cos t_{s}.
\]
However, by Theorem 2$(a),$ this equation has only one double root
$\lambda_{n}^{-}(t_{s})=\lambda_{n}^{+}(t_{s})$ inside $C(2n,3h).$ Thus
\[
\lim_{t\rightarrow t_{s}^{-}}\lambda_{n}^{-}(t)=\lim_{t\rightarrow t_{s}^{+}%
}\lambda_{n}^{-}(t)=\lim_{t\rightarrow t_{s}^{-}}\lambda_{n}^{+}%
(t)=\lim_{t\rightarrow t_{s}^{+}}\lambda_{n}^{+}(t)=\lambda_{n}^{-}%
(t_{s})=\lambda_{n}^{+}(t_{s})
\]
for $s=1,2,...,k.$ In the same way, we prove continuity at $0$ from the right.
The theorem is proved.
\end{proof}

\begin{theorem}
If $n>N(h),$ then the following statements hold:

$(a)$ For each $t\in\lbrack0,h],$ any double eigenvalue $\lambda_{n}^{+}(t)$
is real.

$(b)$ If $\lambda_{n}^{+}(t_{1})$ and $\lambda_{n}^{+}(t_{2})$ are real
numbers, where $0\leq t_{1}<t_{2}\leq h,$ then $\lambda_{n}^{+}(t)$ is real
for all $t\in\lbrack t_{1},t_{2}]$.

$(c)$ There exists at most one $t\in\lbrack0,h),$ such that $\lambda_{n}%
^{+}(t)$ is a double eigenvalue.

$(d)$ Statements $(a)$-$(c)$ remain valid if $\lambda_{n}^{+}(t)$ is replaced
by $\lambda_{n}^{-}(t).$
\end{theorem}

\begin{proof}
$(a)$ Suppose that $\lambda_{n}^{+}(t)$ is an double eigenvalue and
$\lambda_{n}^{+}(t)\notin\mathbb{R}$. Then, by (11), its complex conjugate
$\overline{\lambda_{n}^{+}(t)}$ is also an eigenvalue lying inside the circle
$C(2n,3h)$. This contradicts the statement of Theorem 2$(a)$ that this circle
contain only two eigenvalue, counting multiplicity, in its interior.

$(b)$ Suppose that there exists $t\in(t_{1},t_{2})$ \ such that $\lambda
_{n}^{+}(t)\notin\mathbb{R}$. It follows from Theorem 3 and (11) that
$\left\{  \lambda_{n}^{+}(t):t\in\lbrack t_{1},t_{2}]\right\}  \cup\left\{
\overline{\lambda_{n}^{+}(t)}:t\in\lbrack t_{1},t_{2}]\right\}  $ is a closed
curve lying in the spectrum. However, the spectrum of $L(Q)$ con not contain a
closed curve. Indeed, by (9), the spectrum $\sigma(L(Q))$ is the inverse image
of the interval $[-1,1]$ under the nonconstant entire function $F:\mathbb{C}%
\rightarrow\mathbb{C}.$

$(c)$ Suppose that there exist two double eigenvalues $\lambda_{n}^{+}%
(t_{1})=\lambda_{n}^{-}(t_{1})$ and $\lambda_{n}^{+}(t_{2})=\lambda_{n}%
^{-}(t_{2})$ such that for $t\in(t_{1},t_{2})$ the eigenvalues $\lambda
_{n}^{+}(t)$ and $\lambda_{n}^{-}(t)$ are simple, where $0\leq t_{1}<t_{2}<h$
and $n>N(h)$. By $(a)$, $\lambda_{n}^{+}(t_{1})$ and $\lambda_{n}^{+}(t_{2})$
are real. Without loss of generality, assume that $\lambda_{n}^{+}%
(t_{1})<\lambda_{n}^{+}(t_{2}).$ Then, by Theorem 3 and part $(b)$ the sets
$\left\{  \lambda_{n}^{+}(t):t\in\lbrack t_{1},t_{2}]\right\}  $ and $\left\{
\lambda_{n}^{-}(t):t\in\lbrack t_{1},t_{2}]\right\}  $ are intervals of the
real line containing the interval $[\lambda_{n}^{+}(t_{1}),\lambda_{n}%
^{+}(t_{2})]$. Therefore, for any point $c\in(\lambda_{n}^{+}(t_{1}%
),\lambda_{n}^{+}(t_{2}))$ there exist $d_{1},d_{2}\in(t_{1},t_{2})$ such that
$\lambda_{n}^{+}(d_{1})=\lambda_{n}^{-}(d_{2})=c.$ If $d_{1}=d_{2},$ then
$\lambda_{n}^{+}(d_{1})$ is a double eigenvalue, contradicting the assumption.
If $d_{1}\neq d_{2}$ then using (9), we obtain $\cos\pi d_{1}=\cos\pi d_{2},$
which is impossible because $d_{1},d_{2}\in\lbrack0,h).$

$(d)$ Replacing $\lambda_{n}^{+}(t)$ by $\lambda_{n}^{-}(t)$ throughout the
proofs of $(a)$-$(c),$ gives the proof of $(d).$

The theorem is proved.
\end{proof}

Instead of considering $[0,h)$ and $C(2n,3h),$ we can consider $(1-h,1]$ and
$C(2n+1,3h),$ respectively, and argue in the same way. We first introduce the
following notations.

\begin{notation}
By Theorem 2$(b),$ if $t\in\lbrack1-h,1],$ then the operator $L_{t}\left(
Q\right)  $ has two eigenvalues, counting multiplicities, inside the circle
$C(2n+1,3h)$ for $n>N(h).$ It follows from Summary 1 that, for $t\in
\lbrack1-h,1-\varepsilon_{n}]$ the eigenvalues of $L_{t}(Q)$ lying in
$C(2n+1,3h)$ are simple and real. Moreover, arguing as in Notation 1, we see
that they are$\ \lambda_{n}^{+}(t)$ and $\lambda_{n+1}^{-}(t)$ with
$\lambda_{n}^{+}(t)<\lambda_{n+1}^{-}(t)$ for $t\in\lbrack1-h,1-\varepsilon
_{n}].$ We use this notations for $t\in(1-h,1].$ If, for some $t\in(1-h,1],$
the eigenvalue lying in $C(2n+1,3h)$ is a double eigenvalue, then
$\lambda_{n+1}^{-}(t)=\lambda_{n}^{+}(t).$ If $C(2n+1,3h)$ contains two simple
eigenvalues of $L_{t}(Q)$, then we denote them again by $\lambda_{n}^{+}(t)$
and $\lambda_{n+1}^{-}(t)$ as follows. If they are real, then $\lambda_{n}%
^{+}(t)<\lambda_{n+1}^{-}(t).$ If they are not real, then they are complex
conjugates (see (11)). In this case, we choose the notation so that
$\operatorname{Im}\lambda_{n+1}^{-}(t)<0$ and $\operatorname{Im}\lambda
_{n}^{+}(t)>0.$
\end{notation}

Repeating the proof of Theorem 3 and 4, we obtain the following results for
$t\in\lbrack1-h,1]$ and $n>N(h).$

\begin{theorem}
$(a)$ The eigenvalue $\lambda_{n}^{+}(t)$ of $L_{t}(Q),$ for $n>N(h)$ can be
chosen as continuous function on $[1-h,1].$

$(b)$ For each $t\in\lbrack1-h,1]$ any double eigenvalue $\lambda_{n}^{+}(t)$
is real.

$(c)$ If $\lambda_{n}^{+}(t_{1})$ and $\lambda_{n}^{+}(t_{2})$ are real
numbers, where $1-h\leq t_{1}<t_{2}\leq1,$ then $\lambda_{n}^{+}(t)$ is real
for all $t\in\lbrack t_{1},t_{2}]$.

$(d)$ There exists at most one $t\in\lbrack1-h,1],$ such that $\lambda_{n}%
^{+}(t)$ is a double eigenvalue.

$(e)$ Statements $(a)$-$(d)$ remain valid if $\lambda_{n}^{+}(t)$ is replaced
by $\lambda_{n}^{-}(t).$
\end{theorem}

Thus, it follows from Summary 1, Theorem 3 and Theorem 5$(a)$ that
$\lambda_{n}^{+}(t)$ and $\lambda_{n}^{-}(t)$ are continuous functions at the
intervals $[\varepsilon_{n},1-\varepsilon_{n}],$ $[0,h]$ and $[1-h,1],$
respectively. Moreover, the interval $[\varepsilon_{n},1-\varepsilon_{n}]$
overlaps with subintervals of both $[0,h]$ and $[1-h,1],$ since $\varepsilon
_{n}<h.$ Therefore, we have the following consequence

\begin{corollary}
For $n>N(h)$the functions $\lambda_{n}^{+}(t)$ and $\lambda_{n}^{-}(t)$ are
continuous on $[0,1].$
\end{corollary}

It follows from this corollary that
\begin{equation}
\Gamma_{n}^{-}=\left\{  \lambda_{n}^{-}(t):t\in\lbrack0,1]\right\}  \text{
}\And\text{ }\Gamma_{n}^{+}=\left\{  \lambda_{n}^{+}(t):t\in\lbrack
0,1]\right\}  \tag{24}%
\end{equation}
are continuous curves for $n>N(h).$ Moreover, by Summary 1, a large portion of
these curves lies on the real axis. Denote by $a_{n}^{\pm}$ and $b_{n}^{\pm},$
respectively, the smallest number in $[0,\varepsilon_{n}]$ and the largest
number in $[1-\varepsilon_{n},1]$ such that $\lambda_{n}^{\pm}(a_{n}^{\pm})$
and $\lambda_{n}^{\pm}(b_{n}^{\pm})$ are the real numbers. \ 

\begin{theorem}
If $n>N(h),$ then the following statements hold:

$(a)$ The real parts $\Gamma_{n}^{+}\cap\mathbb{R}$ and $\Gamma_{n}^{-}%
\cap\mathbb{R}$ of the curves $\Gamma_{n}^{+}$ and $\Gamma_{n}^{-}$ are,
respectively $[\lambda_{n}^{+}(a_{n}^{+}),\lambda_{n}^{+}(b_{n}^{+})]$ and
$[\lambda_{n}^{-}(b_{n}^{-}),\lambda_{n}^{-}(a_{n}^{-})].$

$(b)$ If $a_{n}^{+}>0,$ then $a_{n}^{+}=a_{n}^{-},$ $\lambda_{n}^{+}(a_{n}%
^{+})$ is a double eigenvalue, $\lambda_{n}^{+}(a_{n}^{+})=\lambda_{n}%
^{-}(a_{n}^{+}),$ and the parts $\left\{  \lambda_{n}^{+}(r):t\in
\lbrack0,a_{n}^{+})\right\}  $ and $\left\{  \lambda_{n}^{-}(r):t\in
\lbrack0,a_{n}^{+})\right\}  $ of $\Gamma_{n}^{+}$ and $\Gamma_{n}^{-},$
respectively, are curves lying outside the real axis and symmetric with
respect to the real axis.

$(c)$ If $b_{n}^{+}<1,$ then $b_{n}^{+}=b_{n+1}^{-},\lambda_{n}^{+}(b_{n}%
^{+})$ is a double eigenvalue, and the part $\left\{  \lambda_{n}^{+}%
(t):t\in(b_{n}^{+},1]\right\}  $ and $\left\{  \lambda_{n+1}^{-}(r):t\in
(b_{n}^{+},1]\right\}  $ of $\Gamma_{n}^{+}$ and $\Gamma_{n+1}^{-},$
respectively, are curves lying outside the real axis and symmetric with
respect to the real axis.
\end{theorem}

\begin{proof}
$(a)$ If follows from the definition of $a_{n}^{+}$ and $b_{n}^{+}$ that both
$\lambda_{n}^{+}(a_{n}^{+})$ and $\lambda_{n}^{+}(b_{n}^{+})$ are real. Then
by Summary 1, Theorems 4$(b)$ and 5$(c),$ $\lambda_{n}^{+}(t)\in\mathbb{R},$
for all $t\in\lbrack a_{n}^{+},b_{n}^{+}].$ Therefore, using Corollary 1, and
repeating the \textbf{Second proof }of Theorem 1, with $[a_{n}^{+},b_{n}^{+}]$
using instead of $[\varepsilon_{n},1-\varepsilon_{n}],$ we obtain $(a).$

$(b)$ It follows from (11) and Notation 1 that, for $t\in\lbrack0,h],$
$\lambda_{n}^{+}(t)$ is real number if and only if $\lambda_{n}^{-}(t)$ is
real. Therefore, $a_{n}^{+}=a_{n}^{-}.$ Moreover, by the definition of
$a_{n}^{+},$ the curve $\left\{  \lambda_{n}^{+}(t):t\in\lbrack0,a_{n}%
^{+})\right\}  $ lies outside the real axis. Since $\lambda_{n}^{-}(t)$\ is
the complex conjugate of $\lambda_{n}^{+}(t)$ for $t\in\lbrack0,a_{n}^{+}),$
the curve$\left\{  \lambda_{n}^{-}(t):t\in\lbrack0,a_{n}^{+})\right\}  $ is
the symmetric to $\left\{  \lambda_{n}^{+}(t):t\in\lbrack0,t_{0})\right\}  $,
with respect to the real axis.

By the continuities of $\lambda_{n}^{-}(t)$ and $\lambda_{n}^{+}(t)$, we
obtain
\[
\text{ }\lim_{t\rightarrow a_{n}^{+}}\lambda_{n}^{+}(t)=\lambda_{n}^{+}%
(a_{n}^{+}),\text{ }\lambda_{n}^{-}(a_{n}^{+})=\lim_{t\rightarrow a_{n}^{+}%
}\lambda_{n}^{-}(t)=\lim_{t\rightarrow a_{n}^{+}}\overline{\lambda_{n}^{+}%
(t)}=\overline{\lambda_{n}^{+}(a_{n}^{+})}.
\]
This implies that $\lambda_{n}^{+}(a_{n}^{+})=\lambda_{n}^{-}(a_{n}^{+}),$
since $\lambda_{n}^{+}(a_{n}^{+})$ is real. Thus, $\lambda_{n}^{+}(a_{n}^{+})$
is a double eigenvalue.

$(c)$ Instead of Notation 1 using Notation 2 and repeating the proof of $(b),$
we obtain the proof of $(c).$
\end{proof}

\begin{remark}
Note that, in case of the Schr\"{o}dinger operator, the points $\lambda
_{n}^{+}(a_{n}^{+})$ and $\lambda_{n}^{+}(b_{n}^{+}),$ \ where $a_{n}^{+}>0$
and $b_{n}^{+}<1,$ are called complexation points (see Definition 1 and
Theorem 4(c) of [12]). Moreover, arguing as in the proof of Proposition 2 of
[11] and Theorem 3 of [14], one can easily see that these points are spectral
singularities of $L(Q).$
\end{remark}

Now, using Theorem 6, we show that the number of gaps in the real part of the
spectrum of the PT-symmetric Dirac operator depends on the periodic and
antiperiodic eigenvalues. For $n>N(h),$ there are three possible cases:

\begin{case}
$\lambda_{n}^{+}(0)$ is a double eigenvalue, that is, $\lambda_{n}%
^{-}(0)=\lambda_{n}^{+}(0).$ Then by Theorem 4$(a),$ $\lambda_{n}^{+}(0)$ a real.
\end{case}

\begin{case}
$\lambda_{n}^{+}(0)\notin\mathbb{R}$. Then $\operatorname{Im}\lambda_{n}%
^{+}(0)>0$ and $\lambda_{n}^{-}(0)=\overline{\lambda_{n}^{+}(0)}$ (see
Notation 1).
\end{case}

\begin{case}
$\lambda_{n}^{+}(0)$ is a simple real eigenvalue. Then $\lambda_{n}^{-}(0)$ is
also a simple real eigenvalue and $\lambda_{n}^{-}(0)<\lambda_{n}^{+}(0).$
\end{case}

First, let us prove the following result.

\begin{theorem}
$(a)$ In Case 1 and Case 2, the following equality holds:%
\begin{equation}
\left(  \Gamma_{n}^{-}\cap\mathbb{R}\right)  \cup\left(  \Gamma_{n}^{+}%
\cap\mathbb{R}\right)  =[\lambda_{n}^{-}(b_{n}^{-}),\lambda_{n}^{+}(b_{n}%
^{+})]. \tag{25}%
\end{equation}
$(b)$ In Case 3, the intervals $\Gamma_{n}^{-}\cap\mathbb{R}$ and $\Gamma
_{n}^{+}\cap\mathbb{R}$ are separated by the gap $(\lambda_{n}^{-}%
(0),\lambda_{n}^{+}(0)).$
\end{theorem}

\begin{proof}
$(a)$ First consider Case 1. It follows from the definition of $a_{n}^{\pm}$
that, $a_{n}^{\pm}=0.$ Then, by Theorem 6 $(a),$ we have%
\begin{equation}
\Gamma_{n}^{-}\cap\mathbb{R=}[\lambda_{n}^{-}(b_{n}^{-}),\lambda_{n}%
^{-}(0)]\text{ }\And\Gamma_{n}^{+}\cap\mathbb{R}=[\lambda_{n}^{+}%
(0),\lambda_{n}^{+}(b_{n}^{+})]. \tag{26}%
\end{equation}
Since $\lambda_{n}^{-}(0)=\lambda_{n}^{+}(0)$ and $\lambda_{n}^{-}(b_{n}%
^{-})<\lambda_{n}^{+}(b_{n}^{+}),$ (25) holds.

Now, consider Case 2. In this case, $a_{n}^{+}>0.$ Then, by Theorem 6 $(b),$
$\lambda_{n}^{-}(a_{n}^{-})=\lambda_{n}^{+}(a_{n}^{+}).$ Therefore (25)
follows from Theorem 6$(a)$.

$(b)$ Arguing as in the Case 1, we see that (26) holds. Moreover, in Case 3,
we have $\lambda_{n}^{-}(0)<\lambda_{n}^{+}(0).$ Therefore, these intervals
are separated by the gaps $(\lambda_{n}^{-}(0),\lambda_{n}^{+}(0)).$ The
theorem is proved.
\end{proof}

Now, the main result of this paper follows from Proposition 1 and Theorem 7.

\begin{theorem}
$(a)$ The spectrum of $L(Q)$ has finitely many gaps if and only if the number
of the $2\pi$-periodic SR eigenvalues is finite.

$(b)$ The spectrum of $L(Q)$ has finitely many gaps if and only if all but
finitely many $2\pi$-periodic eigenvalues are either double eigenvalues or
nonreal eigenvalues.
\end{theorem}

\begin{proof}
Since there are only three cases, namely Cases 1, 2, and 3, statements $(a)$
and $(b)$ are equivalent. We prove $(b).$ By Theorem 7$(a),$ (25) holds. This
implies that
\[
\lbrack\lambda_{n}^{-}(b_{n}^{-}),\lambda_{n}^{+}(b_{n}^{+})]\subset
\sigma(L(Q)).
\]
On the other hand it follows from the definition of $b_{n}^{-}$ and $b_{n}%
^{+}$ that $b_{n}^{-},$ $b_{n}^{+}$ $\in\lbrack1-\varepsilon_{n},1].$
Therefore, by Notation 2, $\lambda_{n}^{-}(b_{n}^{-})$ and $\lambda_{n}%
^{+}(b_{n}^{+})$ lie inside $C(2n-1,3h)$ and $C(2n+1,3h),$ respectively. In
other words, $\lambda_{n}^{-}(b_{n}^{-})<2n-1+3h$ and $\lambda_{n}^{+}%
(b_{n}^{+})>2n+1-3h.$ Consequently,
\[
(2n-2\varepsilon_{n-1},2n+2\varepsilon_{n})\subset\lbrack\lambda_{n}^{-}%
(b_{n}^{-}),\lambda_{n}^{+}(b_{n}^{+})]\subset\sigma(L(Q)),
\]
since $2\varepsilon_{n-1}+3h<1$ and $2\varepsilon_{n}+3h<1.$ In the same way,
we obtain
\[
(2n+1-2\varepsilon_{n},2n+1+2\varepsilon_{n})\subset\sigma(L(Q)).
\]
Thus, the theorem follows from Proposition 1.
\end{proof}

Since nonreal eigenvalues occur for a large class of non-self-adjoint Dirac
operator, the following consequence of Theorem 8 provides a possibility of
constructing a large class of PT-symmetric finite-zone potentials.

\begin{corollary}
The spectrum of $L(Q)$ has finitely many gaps if all but a finitely many
$2\pi$-periodic eigenvalues are nonreal.
\end{corollary}

Now, to describe briefly the results of Theorem 8, let us introduce the
following notation. Let $\lambda_{k}^{\pm}$ be the eigenvalues of $L_{0}$ (the
periodic eigenvalues) for even $k$ ($k=2n$), or of $L_{1}$ (the antiperiodic
eigenvalues) for odd $k$ ($k=2n+1$) lying in $C(k,3h).$ We denote these
eigenvalues so that $\operatorname{Re}\lambda_{k}^{-}<\operatorname{Re}%
\,\lambda_{k}^{+}.$ If $\operatorname{Re}\,\lambda_{k}^{-}=\operatorname{Re}%
\,\lambda_{k}^{+}$ then $\operatorname{Im}\lambda_{k}^{-}\leq\operatorname{Im}%
\lambda_{k}^{+}.$ Note that if these eigenvalues are not real, then they are
complex conjugates $\,\lambda_{k}^{+}=\overline{\lambda_{k}^{-}}$ (see (11)).
With this notation Theorem 8 can be formulated as follows.

\begin{theorem}
The spectrum of $L(Q)$ has finitely many gaps if and only if there exist
$K>2N(h)$, such that $\operatorname{Re}\gamma_{k}=0,$ for $\left\vert
k\right\vert >K,$ where $\gamma_{k}=\lambda_{k}^{+}-\lambda_{k}^{-}.$
\end{theorem}

The proof is clear, since $\operatorname{Re}\gamma_{k}=0$ if and only if
either Case 1 or Case 2 occurs.

In [13] using asymptotic formulas obtained in [2, 9] for the periodic and
antiperiodic eigenvalues of the Schr\"{o}dinger operator
\[
S(q)=-\frac{d^{2}}{dx^{2}}+q,
\]
we obtained explicit conditions on the Fourier coefficient of $q$ under which
the spectrum of $S(q)$ has finitely many gaps. Djakov and Mityagin obtained
numerous formulas and estimates (see, for instance [3, 4] and the references
therein) for $\gamma_{k},$ $\lambda_{k}^{+}$ and $\lambda_{k}^{-}$. For this,
they proved that for large enough $|n|,$ the number $\lambda=n+z,$ with
$|z|\leq1/4$ is a $2\pi$-periodic eigenvalue of the Dirac operator $L(Q)$ if
and only if $z$ satisfies the basic equation
\[
(z-\alpha_{n}(z))^{2}=\beta_{n}^{+}(z)\beta_{n}^{-}(z),
\]
where $\alpha_{n}(z)$ and $\beta_{n}^{\pm}(z)$ are well defined for large
enough $|n|$ by explicit expressions in terms of the Fourier coefficients of
the potential entries with respect to the system $\{e^{imx},\,m\in
2\mathbb{Z}\}.$ Using Theorem 9 together with the formulas and estimates
obtained in [3, 4] and references therein, one can consider explicit condition
on the potential $Q$ under which the number of gaps in $\sigma(L(Q))$ is finite.

\end{document}